\documentclass[12pt]{amsart}

\usepackage{amsfonts, amssymb, amsmath, amsthm}
\newcommand{\R}{\mathbb{R}}
\newcommand{\N}{\mathbb{N}}
\newcommand{\Q}{\mathbb{Q}}
\newcommand{\Z}{\mathbb{Z}}

\newcommand{\bA}{\mathbf{A}}
\newcommand{\bP}{\mathbf{P}}
\newcommand{\bQ}{\mathbf{Q}}
\newcommand{\bT}{\mathbf{T}}
\newcommand{\bX}{\mathbf{X}}

\newcommand{\A}{\mathcal{A}}

\newcommand{\C}{\mathcal{C}}

\newcommand{\Cech}{{\v Cech}{} }

\newcommand{\dlim}{\varinjlim} 

\newcommand{\larr}{\left( \begin{array}{c}}
\newcommand{\rarr}{\end{array} \right) }
\def\AA{p'(\lambda)}
\def\ie{{\em i.e.,}\ }
\def\eg{{\em e.g.}\ }

\newcommand{\lsqarr}{\left[ \begin{array}{c}}
\newcommand{\rsqarr}{\end{array} \right]}

\newcommand{\arrow}{\rightarrow}

\def\coker{\mathop{\rm coker}\nolimits}

\newtheorem{theorem}{Theorem}
\newtheorem{definition}{Definition}
\newtheorem{lemma}{Lemma}
\newtheorem{corollary}{Corollary}
\newtheorem{proposition}{Proposition}

\begin{document}

\title{Homological Pisot Substitutions and Exact Regularity}
\author{Marcy~Barge, Henk~Bruin, Leslie~Jones, and Lorenzo~Sadun}


\begin{abstract} We consider one-dimensional substitution tiling spaces
where the dilatation (stretching factor) is a degree $d$ Pisot number,
and the first rational \Cech cohomology is $d$-dimensional.
We construct examples
of such ``homological Pisot'' substitutions whose tiling flows do not
have pure discrete 
spectra. These examples are not unimodular, and
we conjecture that the coincidence rank
must always divide a power of the norm of the dilatation.
To support this conjecture, we show that homological Pisot substitutions
exhibit an Exact Regularity Property (ERP), in which the number of occurrences
of a patch for a return length is governed strictly by the length.
The ERP puts strong constraints on the measure of any cylinder set in
the corresponding tiling space.
\end{abstract}

\maketitle

\section{Introduction}\label{sec:intro}

Versions of the Pisot Conjecture occur in number theory
(numeration systems, $\beta$-expansions), discrete geometry, dynamical
systems (construction of Markov partitions, arithmetical coding of
hyperbolic toral automorphisms), and physics (spectral properties of
materials with quasi-periodic atomic structure) - see the survey
\cite{BS}. In all of these settings there is an underlying
substitution and in this context a standard version of the conjecture
is
\\[3mm]
\textbf{Pisot Conjecture}: The $\R$-action (tiling flow) on the tiling
space associated with a one-dimensional substitution of unimodular and
irreducible Pisot type has pure discrete spectrum.
\smallskip

It is known that any such action has a non-trivial discrete part in
its spectrum, see \cite{BT}, and the conjecture is known to be true in the
case of a substitution on two letters, see \cite{BD, HS}. One of
the reasons for interest in this conjecture is that the $\R$-action on
a substitution tiling space has a pure discrete spectrum
(also known as pure point spectrum) if and only if a
one-dimensional quasicrystal, whose atoms are arranged according to
the pattern of any tiling in the tiling space, has pure discrete 
diffraction spectrum, see \cite{Dworkin,LMS}.

A \emph{substitution} $\phi$ is a function from a
finite alphabet $\A$ into the collection of finite
nonempty words from  $\A$, and
extends by concatenation to a map on finite or infinite words. The
\emph{abelianization} of $\phi$ is the
matrix $\bA=\bA_{\phi}$ with $ij$-th entry equal to the
number of $i$'s in $\phi(j)$. The substitution $\phi$ is \emph{primitive}
if the entries of $\bA^m$ are strictly positive for some
$m \ge 1$. In this case $\bA$ has a simple, positive,
Perron-Frobenius eigenvalue $\lambda=\lambda_{\phi}$ which we will call
the \emph{dilatation} of $\phi$.

\begin{definition}[Pisot Substitution]
A \emph{Pisot number} is an algebraic integer greater than $1$, all of
whose algebraic conjugates lie strictly inside the unit circle.
If the dilatation of a primitive substitution is a
Pisot number of algebraic degree $d$, then we say that the substitution
is \emph{Pisot of degree $d$}.
\end{definition}

Let $p(x)=x^d+a_{d-1}x^{d-1} + \cdots + a_0$
be the minimal polynomial of the dilatation.
The \emph{norm} of the dilatation is the product of the
dilatation and its algebraic conjugates and equals $(-1)^da_0$. We
call $a_0$ the \emph{constant coefficient} of the dilatation.

\begin{definition}[Irreducible Substitution]
The minimal polynomial $p(x)$ always divides the characteristic polynomial
of the abelianization of $\phi$.
If the two polynomials are equal,  then we say that the 
substitution is \emph{irreducible}.
\end{definition}

A substitution $\phi$ generates a collection of
bi-infinite \emph{allowed words}, namely those $\bar{w}=\ldots
w_{-1}w_{0}w_1\ldots$
for which each finite subword $w_i \ldots w_{i+j}$ is a
subword of $\phi^m(a)$ for some $m \in \N$ and $a \in \A$.
To each $a \in \A$ we associate a \emph{prototile} $P_a$
whose length equals the $a$-th component of the
left Perron-Frobenius eigenvector of $\bA_\phi$.
To each allowed word $\bar{w}$ we can associate continuously many
\emph{tilings} $\bT$, replacing each letter $w_i$ by
a copy of the corresponding prototile, and putting an \emph{origin}
somewhere in the obtained concatenation of tiles.
The \emph{tiling space}, $\Omega_\phi$, is the collection of all such tiles,
and it carries  a natural \emph{tiling flow}, $\bT = \{T_i\} \mapsto
\bT-t :=\{T_i-t\}$.

When studying substitutions, it is natural to require that the
substitution be irreducible and Pisot. However, irreducibility is not
a natural condition for tiling {\em spaces}, insofar as different
substitutions may give rise to the same tiling space. The
substitutions $1 \mapsto 21$, $2 \mapsto 1$ and $1 \mapsto 32$, $2
\mapsto 31$, $3 \mapsto 2$ both generate the Fibonacci tiling, but the
first is irreducible while the second is not. We therefore introduce a
topological condition,
noting that if two
one-dimensional substitution tiling spaces are homeomorphic, then the
tiling flow on one of them has pure discrete spectrum if and only if
the tiling flow on the other does. (This is a consequence of the
rigidity result of \cite{BSw}.) We also consider relaxing the
requirement of unimodularity, as there are many instances where pure
discreteness is known to hold when the Pisot dilatation has norm other
than $\pm1$ (in particular, for a large family of substitutions
associated with $\beta$-expansions - see \cite{BBK}).


\begin{definition}[Homological Pisot Substitution]
Let $\phi$ be a degree $d$ Pisot substitution.
If the dimension of the
first rational \Cech cohomology of the tiling space is $d$, then we
say that the substitution is \emph{homological Pisot of degree $d$}.
\end{definition}

The dilatation being a Pisot number is a
necessary and sufficient condition for the tiling flow spectrum to have a discrete component, 
see \cite{solomyak1}.
However, it is easy to find examples of substitutions with
Pisot dilatations whose spectra contain a continuous component. The simplest
example is the Thue-Morse substitution $1 \mapsto 12$, $2\mapsto 21$. To
eliminate the continuous spectrum, we must constrain the system
further, either through a combinatorial condition (irreducibility) or
through a topological condition (dimension of the first cohomology).

In this paper we study the consequences of a substitution being
homological Pisot, including (but not limited to) the question of
whether the tiling flow associated with such a substitution must
have pure discrete spectrum.
Note that the
homological Pisot condition is neither stronger nor weaker than the
irreducible Pisot condition.  It is easy to construct homological Pisot
substitutions that are not irreducible -- just rewrite any homological Pisot
substitution in terms of collared tiles, see \cite{AP}.
It is also easy to find examples
of irreducible Pisot substitutions whose first cohomology has dimension greater
than $d$ (\eg the substitution, $1\mapsto 21112$,
$2\mapsto 121$, is an irreducible degree $2$ Pisot substitution with an
asymptotic cycle, whose tiling space has $3$-dimensional first cohomology).
Nevertheless, the two notions are closely related: an irreducible Pisot substitution
that does not have asymptotic cycles must be homological Pisot. Also, if the
characteristic polynomial of the abelianization of a homological Pisot substitution
does factor over $\Z$, say as $p(x)q(x)$, with $p(x)$ the minimal polynomial of the dilatation,
then all roots of $q(x)$ are zero or roots of unity.

\subsection*{The Coincidence Rank Conjecture}
Every one-dimensional substitution
tiling space has a \emph{maximal equicontinuous factor} consisting of
a Kronecker flow on a torus or solenoid. The continuous map factoring
the tiling flow onto its maximal equicontinuous factor is called
\emph{{geometric realization}} and also factors the substitution
homeomorphism onto a hyperbolic automorphism of the torus or
solenoid. Geometric realization has the following properties
(\cite{BK, BBK}): it is nontrivial if and only if the
substitution is Pisot and, if the substitution is Pisot, then geometric
realization is boundedly finite-to-one and almost everywhere
$cr$-to-one for some positive integer $cr$ called the \emph{coincidence
rank} of the substitution. Furthermore, the tiling flow has pure
discrete spectrum if and only if geometric realization is a.e.\
$1$-to-$1$ (that is, if $cr=1$), in which case
geometric realization is a continuous measurable isomorphism from the
tiling flow onto a Kronecker action.


We can obtain a topological version of the Pisot Conjecture
by dropping the unimodular assumption
(\ie norm $=\pm 1$) and replacing
``irreducible Pisot'' by ``homological Pisot'', leading to the
conjecture:
$$
\text{\em
The tiling flow of a homological Pisot substitution has
pure discrete spectrum.}
$$
This conjecture is {\bf false}.  In the last
section of this paper we give several examples of homological Pisot substitutions with coincidence rank
$3$ and norm divisible by $3$, and in fact, such examples generalize
to every algebraic degree.

\begin{theorem}\label{nogo}  There are homological Pisot
  substitutions of every algebraic degree whose tiling flows do not
  have pure discrete spectrum.
\end{theorem}

That the coincidence rank in these
examples divides the norm is not a coincidence and we are led to the
following conjecture.
\\[3mm]
\textbf{Coincidence Rank Conjecture}: The coincidence rank of a
homological Pisot substitution divides a power of the norm.  In
particular, the tiling flow of a unimodular homological Pisot
substitution has pure discrete spectrum.
\smallskip

\begin{theorem}\label{dim1} The Coincidence Rank Conjecture is true
if the degree of the substitution is one.
\end{theorem}

\subsection*{Exact Regularity}
A \emph{patch} $\bP$ of a tiling $\bT$ is a collection
of contiguous tiles of $\bT$ and the \emph{support} of a patch
is the union of all the tiles in the patch. A \emph{vertex} is any
boundary point of two adjacent tiles in a tiling or patch.

The {\em tile lattice}, $\Gamma$, for the tiling space $\Omega_{\phi}$
is the additive subgroup of $\R$ generated by the lengths of the tiles
of $\Omega_{\phi}$.  A {\em return length} for the tiling space
$\Omega_{\phi}$ is a number $L>0$, such that $L$ is the distance between corresponding points 
in separate occurrences of the same tile type in $\bT$ for some $\bT\in \Omega_{\phi}$.  
The {\em return lattice} for $\Omega_{\phi}$ is the
subgroup of the tile lattice $\Gamma$ generated by $\{\lambda^{-i}L:L$
is a return length and $i \ge 0 \}\cap \Gamma$,
where $\lambda$
is the dilatation of $\phi$. These objects are called lattices, despite
typically being
dense in $\R$, as they are projections to $\R$ of lattices in $\Z^d$,
where $d$ is the number of letters in the substitution.
The tile and return lattices for a `proper' substitution are the same,
and a power of any primitive substitution can be `rewritten' to a proper
substitution -  see \eg \cite{BDproper}.

\begin{definition}[The Exact Regularity Property]
  Let $\lambda$ denote the dilatation of the
  substitution $\phi$ and let $L>0$ be in the return lattice for
  $\Omega_{\phi}$. We say that the tiling space $\Omega_\phi$ exhibits
  the \emph{Exact Regularity Property} $($ERP$)$ if for each patch
  $\bP$ of any tiling in $\Omega_\phi$ there is a length $L'$
  and a linear functional
  $\mathbf{N}_{\bP}:\Q(\lambda)\to \Q$ such
  that: if $\bQ$ and $\bQ+\tau$ are patches that occur
  in any tiling $\bT$ in $\Omega_{\phi}$ and $\bQ$ has a vertex
  $x_0$ with $[x_0-L',x_0+L']$ contained in the support of
  $\bQ$, then the number of occurrences of $\bP$ in
  $\bT$ between $x_0$ and $x_0+\tau$ is exactly
  $\mathbf{N}_{\bP}(\tau/L)$. That is, if $x_1$ is a vertex of
  $\bP$, then $\{t:0\le t<\tau$ and
  $\bP+x_0-x_1+t\subset \bT\}$ has cardinality
  $\mathbf{N}_{\bP}(\tau/L)$.
\end{definition}

\begin{theorem}[Exact Regularity]\label{ERP}
  Homological Pisot substitution spaces exhibit the Exact Regularity
  Property.  Moreover, given any patch $\bP$, if
  $\mathbf{N}_{\bP}$ is expressed in the form
  $\mathbf{N}_{\bP}(\sum_{i=0}^{d-1}c_i\lambda^i)=
  \sum_{i=0}^{d-1}\alpha_ic_i$, with $\alpha_i,c_i\in\Q$, then
  $\alpha_i\in \Z[1/a_0]$ for $i=0,\ldots,d-1$, where $a_0$ is
  the constant coefficient of $\lambda$.
\end{theorem}

\begin{theorem}\label{ERPconverse} Any Pisot substitution that exhibits
the Exact Regularity Property is a homological Pisot substitution.
\end{theorem}

If $\phi$ is irreducible Pisot, then the tiling flow
$\bT=\{T_i\} \mapsto \bT-t :=\{T_i-t\}$
preserves a unique measure $\mu$.
Using the ERP, we derive restrictions on the measures of measurable sets
in the tiling space.

\begin{theorem}\label{rationalmeasure} Let $\phi$ be a homological
  Pisot substitution of algebraic degree $d$ and constant coefficient
  $a_0$. Suppose that the tile lattice and the return lattice for $\phi$ are the same. If some finite
disjoint union of cylinder sets in the tiling space
  $\Omega_\phi$ has rational measure $n/m$, with $n$ and $m$
  relatively prime, then $m$ divides $d \cdot a_0^k$ for some positive
  integer $k$.
\end{theorem}

If $d=1$, it is possible to construct a disjoint union of cylinder sets of
measure $1/cr$ (Proposition~\ref{fraction}).
Theorem~\ref{rationalmeasure} then implies Theorem~\ref{dim1}.

It is even possible to extend the idea of the ERP to one-dimensional
tiling spaces that do not come from a substitution, (\cite{Sadun2}).
If the first rational \Cech cohomology of a tiling
space $\Omega$ is $k$-dimensional, then we can find $k$ different
collections of patches (call them patches of type 1, 2, $\ldots$,
$k$), such that for any other patch $\bP$ there is a length $L'$ with
the following property: If $\bQ$ and $\bQ+\tau$ are
patches that occur in any tiling $\bT$ in $\Omega_\phi$ and $\bQ$
has a vertex $x_0$ with $[x_0-L',x_0+L']$ contained in the support of
$\bQ$, then the number of occurrences of $\bP$ in $\bT$
between $x_0$ and $x_0+\tau$ is a rational linear function of $n_1,
\ldots, n_k$, where $n_i$ is the number of occurrences of patches of
type $i$ between $x_0$ and $x_0 + \tau$. When the tiling space
comes from a substitution and $k$ is the algebraic degree of
the dilatation, then the $i$-th class of patches
is associated with tiles of length $L \lambda^{i-1}$.

\subsection*{Organization of the paper}
In Section~\ref{sec:background},
we develop necessary background and notation. In Section~\ref{sec:exact_regularity}, we explore the ERP and prove Theorems~\ref{ERP}, \ref{ERPconverse}
and \ref{rationalmeasure}. The proofs of these theorems do not use the
fact that the dilatation is a Pisot number, and the results of
Section~\ref{sec:exact_regularity} apply to all substitutions for which the dimension of the first \v{C}ech cohomology equals the algebraic degree of the dilatation.

In Section~\ref{sec:d=1},
we study $d=1$ homological Pisot substitutions and prove
Theorem~\ref{dim1}.  We also show, regardless of the
norm:
\begin{theorem}\label{not2} If $\phi$ is a homological Pisot
  substitution with $d=1$, then the coincidence rank is not $2$.
\end{theorem}
Finally, in Section~\ref{sec:examples}
we show how to construct homological Pisot
substitutions of arbitrary $d$ with $cr=3$, thereby proving Theorem~\ref{nogo}.

\section{Background}\label{sec:background}
A \emph{substitution} is a function $\phi:\A \rightarrow \A^*$ from a
finite \emph{alphabet} $\A$ into the collection $\A^*$ of finite
nonempty words in $\A$. A substitution extends by concatenation to a
map on finite or infinite words and can be iterated: $\phi^m$ will
stand for $\phi \circ \phi \circ \cdots \circ \phi$, $m$ times.  The
\emph{abelianization} or \emph{substitution matrix}  of $\phi$ is the
matrix $\bA=\bA_{\phi}$ with $ij$-th entry equal to the
number of $i$'s in $\phi(j)$. The substitution $\phi$ is \emph{primitive} if the entries of $\bA^m$ are strictly positive for some
$m \ge 1$. In this case $\bA$ has a simple, positive,
Perron-Frobenius eigenvalue $\lambda=\lambda_{\phi}$ which we will call
the \emph{dilatation} of $\phi$.  If $\phi$
is primitive, we will use $\omega^l$ and $\omega^r$ to denote left
and right positive eigenvectors of $\bA$.



To construct the \emph{tiling space} associated with
the primitive substitution $\phi$ we first form the collection
$\bX_{\phi} \subset \A^{\Z}$ consisting of all
bi-infinite \emph{allowed words} for $\phi$: $\bar{w}=\ldots
w_{-1}w_{0}w_1\ldots \in \bX_{\phi}$ if and only if for each
$i \in \Z$, and each $j \ge 0$, there is an $m \in \N$ and $a \in
\A$ so that the word $w_i \ldots w_{i+j}$ is a factor
(subword) of $\phi^m(a)$. If $\A=\{1,\ldots,n\}$ and the left
eigenvector of $\bA$ is $\omega^l=(\omega_1,\ldots,\omega_n)$,
the intervals $P_i=[0, \omega_i]$, $i=1,\ldots,n$, are called
\emph{prototiles} (in case $\omega_i=\omega_j$ for $i\ne j$, we label
$P_i$ and $P_j$ so as to make them distinct). The
tiling space, $\Omega_{\phi}$, associated with $\phi$ is the
collection of all tilings of $\R$ by translates of prototiles
following patterns of allowed bi-infinite words.

A \emph{patch} $\bP$ of a tiling $\bT$ is a subcollection
of contiguous tiles of $\bT$ and the \emph{support} of a patch
is the union of all the tiles in the patch. We will denote the
diameter of the support of $\bP$ by $|\bP|$ and call
this the \emph{length of $\bP$}.

There is a natural \emph{tiling flow}, $\bT=\{T_i\} \mapsto
\bT-t :=\{T_i-t\}$.  We put a metric on $\Omega_{\phi}$
with the property that $\bT$ and $\mathbf{T'}$ are close if
small translates of $\bT$ and $\mathbf{T'}$ agree in a large
neighborhood of the origin.
%
%
If $\phi$ is primitive then, with this metric, $\Omega_{\phi}$ is a
continuum (compact and connected) and the tiling flow is minimal and
uniquely ergodic.

The substitution $\phi$ is \emph{aperiodic} provided there are no
flow-periodic tilings in $\Omega_\phi$.
There is also a $\Z$-action on $\Omega_{\phi}$ induced by
substitution. To define this, suppose that $i \in \A$ and
$\phi(i)=i_1\cdots i_k$. Define $\Phi$ on prototiles by
\begin{equation*}\label{Phi}
\Phi(P_i):=\{P_{i_1},
P_{i_2}+\omega_{i_1},\ldots,P_{i_k}
+\omega_{i_1}+\omega_{1_2}+\cdots +\omega_{i_{k-1}}\}.
\end{equation*}
Extend this to tiles by
$\Phi(P_i+t):=\Phi(P_i)+\lambda t$,
and extend to tilings by $\Phi(\{T_i\}):=\cup \Phi(T_i)$.
As long as $\phi$ is primitive and
aperiodic, $\Phi$ is a homeomorphism, see \cite{recog, solomyak}.
We also recall the machinery of pattern-equivariant cohomology with
rational coefficients from \cite{Kellendonk, KP, Sadun}.
A tiling $\bT \in \Omega_\phi$ gives the real line the
structure of a CW complex, with the vertices serving as $0$-cells and
the tiles serving as $1$-cells.
\begin{definition}
  A rational $0$-cochain is said to be \emph{pattern-equivariant with
    radius $R$} if, whenever $x$ and $y$ are vertices of $\bT$ and the
  radius $R$ neighborhoods $\mathbf{B}_R[\bT-x]=\mathbf{B}_R[\bT-y]$,
  the cochain takes the same values at $x$ and $y$. A \ $0$-cochain is
  \emph{pattern-equivariant} if it is pattern-equivariant with radius
  $R$ for some finite $R$. Pattern-equivariant $1$-cochains are
  defined similarly -- their values on a $1$-cell depend only on the
  pattern of the tiling out to a fixed finite distance around that
  $1$-cell.
\end{definition}
If $\beta$ is a rational pattern-equivariant $0$-cochain, its
coboundary, $\delta(\beta)$, is a rational pattern-equivariant
$1$-cochain, and we define the rational first pattern-equivariant
cohomology of $\bT$ to be the cokernel of the coboundary map $\delta$.
A priori this would seem to depend on $\bT$, but this cohomology is
the same for all $\bT \in \Omega_\phi$ and is isomorphic to $\check
H^1(\Omega_\phi, \Q)$, see \cite{Kellendonk, KP, Sadun}.


Finally, we recall a procedure for computing the first \v Cech
cohomology of a tiling space, using the machinery of
\cite{BDcohomology}. To each primitive one-dimensional substitution
$\phi$ on $n$ letters, one can associate a graph $G$.  This graph has
two kinds of edges. There is one edge $e_i$ for each letter $a_i$ of
the alphabet, and one edge $v_{ij}$ for each two-letter word
$a_ia_j$. The edge $e_i$ has length $\omega_i-\epsilon$ and represents
the bulk of a tile of type $a_i$, while $v_{ij}$ has length $\epsilon$
and represents the transition from the end of a tile of type $i$ to
the beginning of a tile of type $j$. We identify the end of $e_i$ with
the beginning of $v_{ij}$, and the end of $v_{ij}$ with the beginning
of $e_j$. There is also a sub-complex $G_0$ obtained from just the $v$
edges.

After applying a small homotopy, we may assume that substitution maps
$G_0$ to itself. If $\phi(a_i)$ ends with $a_k$, and if $\phi(a_j)$
begins with $a_\ell$, then $\phi(v_{ij})=v_{k\ell}$. Let $G_0^{ER}$ be
the eventual range of $G_0$ under this map. Since substitution
permutes the edges of $G_0^{ER}$, we can replace $\phi$ with a power
that fixes each edge of $G_0^{ER}$. Then there is an exact sequence,
\cite{BDcohomology}
\begin{equation}\label{exact sequence}
0 \to \tilde H^0(G^{ER}_0) \to \dlim {\mathbf A^T} \to
\check H^1(\Omega_\phi) \to H^1(G^{ER}_0) \to 0
\end{equation}
that computes $\check H^1(\Omega_\phi)$. Furthermore, each map in this
exact sequence commutes with substitution, so the image of $\tilde
H^0(G_0^{ER})$ lies in the $+1$ eigenspace of $\mathbf A^T$. Since the
dilatation and its algebraic conjugates are all eigenvalues of
$\mathbf A^T$, the dimension of $\check H^1(\Omega_\phi,\Q)$ is at
least $d$, and equals $d$ only if three conditions are met: (1) the
only eigenvalues of $\mathbf A^T$ are 0, 1, $\lambda$, and the
algebraic conjugates of $\lambda$, (2) the algebraic multiplicity of
the eigenvalue $1$ is one less than the number of components of
$G_0^{ER}$, and (3) $G_0^{ER}$ has no loops.

\section{Exact Regularity}\label{sec:exact_regularity}

Given a substitution $\phi$ with dilatation $\lambda$ of degree $d$,
let us select a return length $L$. Then for every (by minimality of
the flow) $\bT \in \Omega_{\phi}$, $\bT$ and $\bT - L$ have nonempty
intersection.  It follows that for any $L'$, there exists a $k$ such
that $\phi^k (\bT)$ and $\phi^k(\bT) - \lambda^k L$, and hence
$\bT$ and $\bT-\lambda^kL$, intersect in a patch of
length at least $L'$.

As $1,\lambda,\lambda^2,\ldots,\lambda^{d-1}$ are linearly independent
over $\Q$, the length of each patch $\bP$ can be expressed
uniquely in the form $|\bP|=L\sum_{i=0}^{d-1}c_i\lambda^i$ with
$c_i=c_i(|\bP|)\in \Q$. Given any tiling $\bT\in
\Omega_{\phi}$, let $\xi_i$ be the pattern equivariant $1$-cochain on
$\bT$ defined by $\xi_i(T):=c_i(|T|)$ for each $T\in \bT$.

\begin{lemma}\label{lemL1} The cohomology classes $[\xi_0],\ldots,[\xi_{d-1}]$
  are linearly independent in the pattern-equivariant cohomology of
  $\Omega_{\phi}$.
\end{lemma}

\begin{proof}

  Suppose that $\alpha=\sum \alpha_i\xi_i = \delta(\beta)$, where
  $\beta$ is a pattern-equivariant $0$-cochain and each $\alpha_i$ is
  rational. For large enough $k$, $\lambda^k L$ will be a return
  length between patches of size greater than twice the radius of
  $\beta$.  Therefore, $\alpha$ applied to such a return patch of length
$\lambda^{k+i} L$ must be
  zero for $i\geq 0$. But $\alpha$ applied to a patch depends only on
the length of the patch, so $\alpha$ applied to {\em any} patch of length
$\lambda^{k+i}L$ must be zero.
As $\lambda^d + a_{d-1} \lambda^{d-1} + \cdots +
  a_0 = 0$, after division by $\lambda$, we have $a_0 \lambda^{-1} =
  -(\lambda^{d-1} + a_{d-1} \lambda^{d-2} + \cdots + a_1).$ Taking the
  $(k-i)th$ power of the last equation, we have $a_0^{k-i}
  \lambda^{i-k}$ as a polynomial in $\lambda$ with integer
  coefficients.  Hence, $a_0^{k-i} L \lambda^i$ is an integer linear
  combination of $\lambda^k L$, $\lambda^{k+1}L$, $\ldots
  \lambda^{k+d-1}L$, so $\alpha_i$ must be zero.
\end{proof}

If $\phi$ is a homological Pisot substitution,
$\{[\xi_0],\ldots,[\xi_{d-1}]\}$ is a basis for the first cohomology
of any tiling in $\Omega_{\phi}$.  We use this to prove the Exact
Regularity Property (Theorem \ref{ERP}) and its converse
(Theorem~\ref{ERPconverse}).


\begin{proof}[Proof of Theorem~\ref{ERP}]
  Let $\alpha$ be a pattern-equivariant $1$-cochain on $\bT$
  (which we call an \emph{indicator cochain}) that
  evaluates to $1$ at each occurrence of the patch $\bP$ and is
  identically zero away from $\bP$. That is, choose a tile
  $T\in\bP$ and define $\alpha$ by: if $T'$ is any tile of
  $\bT$, then $\alpha(T')=1$ if there is a $t\in \R$ such that $T+t=T'$ and
  $\bP+t\subset \bT$, and $\alpha(T') =0$ otherwise. We must
  then have $\alpha = \delta(\beta) + \sum \alpha_i \xi_i$, where
  $\beta$ is a pattern-equivariant $0$-cochain of some radius $r$ and
  the coefficients $\alpha_i$ are rational. Let
  $L'=max\{r,|\bP|\}$ and suppose that patch $\bQ\subset
  \bT$ has a vertex $x_0$ with $ [x_0-L',x_0+L']$ contained in the
  support of $\bQ$. Suppose also that $\bQ+\tau\subset
  \bT$. Since $\beta(x_0+ \tau)=\beta(x_0)$, $\alpha([x_0, x_0+\tau])
  = \sum \alpha_i \xi_i ([x_0,x_0+\tau]) = \sum \alpha_i
  c_i(\tau)$. On the other hand, $\alpha([x_0, x_0+\tau])$ is the
  number of occurrences of $P$ between $x_0$ and $x_0+\tau$, this
  number being unambiguous (\ie independent of choice of $T\in
  \bP$) since $L'\ge |\bP|$.

  If $L''$ is any return length, then for large $k$, $\lambda^{k+i}
  L''$ is a return length between patches of size greater than $L'$
  for all $i\ge 0$.  Then $\sum \alpha_i \xi_i=\alpha-\delta(\beta)$
  applied to a patch of length $\lambda^{k+i} L''$ yields an integer.
  However, $a_0^k L''$ is an integer linear combination of $\lambda^k
  L''$, $\lambda^{k+1} L'', \ldots, \lambda^{k+d-1} L''$.  This
  implies that $\sum \alpha_i \xi_i$ applied to a patch of length
  $L''$ yields an integer divided by a power of $a_0$.  Since $\sum
  \alpha_i \xi_i$ applied to any return word yields
  an integer divided by a power of $a_0$, and since $L \lambda^i$ is
  in the return lattice, all coefficients $\alpha_i$ must be integers
  divided by powers of $a_0$.
\end{proof}


\begin{proof}[Proof of Theorem~\ref{ERPconverse}]
  Let $\alpha$ be an indicator cochain for a patch $\bP$ (as in
  the proof above). There are then $\alpha_i \in \Q$ so that
  $\gamma:=\alpha-\sum \alpha_i \xi_i$ vanishes on chains of the form
  $[x_0,x_0+\tau]$ with $x_0$ a vertex in a patch $\bQ$ of
  $\bT$, $[x_0-L',x_0+L']$ contained in the support of $\bQ$,
  and $\bQ+\tau\subset\bT$. Let $\bQ$ and $x_0$ be such
  a patch and vertex and define the $0$-cochain $g$ by
  $g(x):=\gamma([x_0,x])$ for each vertex $x$ of $\bT$ (here $[x_0,x]$
  means $-[x,x_0]$ for $x<x_0$, and we take $g(x_0)=0$). Then $g$ is
  pattern-equivariant and $\delta g=\gamma$. Since the indicator
  cochains span the pattern-equivariant $1$-cochains, the $[\xi_i]$ form
  a basis for the first cohomology of $\bT$.
\end{proof}

Having established constraints on the number of occurrences of
$\bP$  for any return length $L$ (or at least $\lambda^k L$ for  sufficiently large $k$), we establish constraints on the measure of
the cylinder set of $\bP$ with respect to the unique
translation-invariant measure $\mu$ on $\Omega_{\phi}$ .  Let
$S_{\bP}$ be the set of all tilings in which the origin lies
inside a $\bP$ patch.

\begin{proposition}
  Suppose that $\phi$ is a homological Pisot substitution of degree
  $d$ for which the tile and return lattices are the same, and let $\AA = d \lambda^{d-1} + \sum_{i=1}^{d-1} i a_i
  \lambda^{i-1}$ be the derivative of the minimal polynomial of the
  dilatation $\lambda$ of $\phi$, evaluated at $\lambda$.
  Then $\displaystyle
  \mu(S_{\bP})=\frac{q_{\bP}(\lambda)}{a_0^k \AA}$,
  where $k$ is an integer and $q_{\bP}(\lambda)$ is a
  polynomial in $\lambda$ with integer coefficients.
\end{proposition}

\begin{proof}
  Let $\bP$ be a patch based on the word $w=w_1\cdots w_l$.  Given any
  patch $\bQ$, let $S_{\bQ}^j$ denote the (partial)
  cylinder set consisting of all tilings for which the origin not only
  lies in a $\bQ$ patch, but in the $j$th tile of the
  $\bQ$ patch. We may then express $S_{\bP}$ as a (measurably)
  disjoint union of sets of the form $S_{\mathbf {Q}}^{l+1}$ where
  $\bQ$ has underlying word $uw_iv$, $u$ and $v$ $l$-letter
  words, $i\in \{1, \ldots,l\}$.  For $\bQ$ based on the word
  $uw_iv$ and $m$ such that the length of $\phi^m(w_i)$ is a return
  length, we take $L=\lambda^m |w_i|$. To compute the measure of
  $S^{l+1}_{\bQ}$, we merely count how many times $\bQ$
  occurs in an interval of length $L \lambda^k$, divide by
  $\lambda^k$, take the limit as $k \to \infty$, and then multiply by
  $|w_i|/L$.  To find the limit, we write $\lambda^k =
  \sum_{i=0}^{d-1} c_{k,i} \lambda^i$ and take the limit of $\sum
  \alpha_i c_{k,i} / \lambda^k$. The limit of $\vec r = \lim_{k \to
    \infty} (c_{k,0}, \ldots, c_{k,d-1})^T / \lambda^k$ is a right
  eigenvector of the companion matrix
\begin{equation*}
  C = \begin{pmatrix} 0 & 0 & \cdots & 0 & -a_0 \cr
    1 & 0 & \cdots & 0 & - a_1 \cr
    0 & 1 & \cdots & 0 & -a_2 \cr
    \vdots &  & \ddots & \vdots & \vdots \cr
    0 & 0 & \cdots & 1 & -a_{d-1}
\end{pmatrix},
\end{equation*}
normalized so that $(1, \lambda, \ldots, \lambda^{d-1}) \vec r = 1$.
This eigenvector is
\begin{equation*}
  \vec r = \frac{1}{\AA} \begin{pmatrix}
    \lambda^{d-1} + a_{d-1} \lambda^{d-2} + \cdots + a_1  \cr
    \lambda^{d-2} + a_{d-1} \lambda^{d-3} + \cdots + a_2 \cr
    \lambda^{d-3} + a_{d-1} \lambda^{d-4} + \cdots + a_3 \cr
    \vdots \cr
    \lambda + a_{d-1} \cr
    1
\end{pmatrix}.
\end{equation*}
Since each entry of $\vec r$ is a polynomial in $\lambda$ divided by
$\AA$, and since each $\alpha_i$ is an integer divided by a power of
$a_0$, the measure of $S^{l+1}_{\bQ}$ is of the indicated form.
It follows that $S_{\bP}$ has the desired form as well.
\end{proof}

\begin{lemma}\label{lemma10}
  Let $\phi$ be a homological Pisot substitution of degree $d$ and dilatation $\lambda$ for which the tile and return lattices are the same. If some
  cylinder set in $\Omega_{\phi}$ has rational measure $n/m$, with $n$
  and $m$ relatively prime, then $m$ divides $a_0^k \gcd(\AA)$ for
  some $k \in \N$, where $\gcd(\AA)$ is the greatest common divisor of
  the coefficients of $\AA$ and $a_0$ is the constant coefficient of $\lambda$.
\end{lemma}

\begin{proof}
  If $\displaystyle \frac{q_{\bP}(\lambda)}{a_0^k\AA} =
  \frac{n}{m}$ , then $m q_{\bP}(\lambda) = a_0^k \AA n$.  But
  this means that $m$ divides every coefficient of $a_0^k \AA$, and so
  divides $a_0^k \gcd(\AA)$.
\end{proof}

Theorem~\ref{rationalmeasure} is an immediate corollary of Lemma~\ref{lemma10}.

\section{Substitutions with $d=1$}\label{sec:d=1}

If $d=1$, then the dilatation is an integer $N$, and the norm is $N$
itself. Under these circumstances, we can replace the substitution
with an `equivalent' substitution that has constant length
using the following technique.

First scale the left Perron-Frobenius eigenvector so that all entries
are integers whose greatest common factor is one. That is, choose all
of the tiles to have integral length. Then subdivide each tile into
smaller pieces, each of length one. The substitution, written in terms
of these pieces, will have constant length $N$.
If the resulting substitution has return lattice $h \Z$, with $h>1$, we can regroup
these pieces into tiles of size $h$ and rescale by $h$. The resulting
substitution has constant length, with both the tile and return
lattices equaling $\Z$, and is called the \emph{pure core} of the
original substitution.
This new substitution
and the original substitution have conjugate tiling flows, hence their coincidence ranks
are the same and one is homological Pisot if and only if the other is.

Consider the substitution $\phi$ of constant length $N$ acting on the
alphabet $\A_\phi$.  We say that two words $w_1 w_2 w_3 \dots$ and
$v_1 v_2 v_3 \dots$ are \emph{coincident} if there exists a $k$ such
that $w_k = v_k$.  We say that the letters $a$ and $b$ in $\A_\phi$
are \emph{eventually coincident} if there exist a $k$ and $n$ such
that $\phi^n(a)_k = \phi^n(b)_k$, where $\phi^n(a)_k$ denotes the
$k$th letter of $\phi^n(a)$.  We say that the letters $a$ and $b$ are
\emph{strongly coincident} if for each $n \geq 0$ and each $i \in \{1,
\dots, N^n \}$, the $i$th letters of $\phi^n(a)$ and $\phi^n(b)$ are
eventually coincident.

\begin{proposition}\label{factor}
  Let $\phi$ be a substitution of constant length $N$ and let $\phi'$
  be the substitution obtained by identifying the letters in $\A_\phi$
  that are strongly coincident. If $\phi$ is a homological Pisot
  substitution, then so is $\phi'$.
\end{proposition}

\begin{proof}
  If $\phi$ is a homological Pisot substitution, then $\phi$ has the
  Exact Regularity Property by Theorem~\ref{ERP}. Since every patch
  $\bP'$ in a tiling of $\Omega_{\phi'}$ corresponds to a finite set of
  patches $\{\bP_1, \ldots, \bP_n\}$ in $\Omega_\phi$, and since each of
  these patches $\bP_i$ is governed by the ERP, we will show that $\bP'$
  is governed by the ERP, implying that $\phi'$ is a homological Pisot
  substitution.

  The patches $\bP_i$ exhibit the ERP with different lengths $L'_i$. Let
  $L'=\max\{L'_i\}$.
  Pick any patch ${\bf S}'$ in $\Omega_{\phi'}$ with a vertex that is
  a distance at least $L'$ from each end. ${\bf S}$ corresponds to a finite
  set $\{ {\bf S}_1, \ldots, {\bf S}_m\}$ of patches in $\Omega_\phi$,
  each with a vertex of distance at least $L'$ from the end. By the ERP
  for each ${\bf P}_i$, between any two successive ${\bf S}_i$'s, there are
  exactly the right number of $\bP_1$'s, the right number of $\bP_2$'s,
  etc, so between any two ${\bf S}'$'s that correspond to the same ${\bf S}_i$,
  there are exactly the right number of $\bP'$'s. The problem is that
  different occurrences of ${\bf S}'$ may correspond to different ${\bf S}_i$'s.

  The discrepancy in how many extra $\bP'$s occur between an ${\bf
    S}_i'$ and an ${\bf S}_j'$ depends only on $i$ and $j$, since the
  discrepancy between any two ${\bf S}_i'$s, or any two ${\bf S}_j'$s,
  is zero.  In particular, there exist numbers $z_1,\ldots,z_m$ so
  that there are exactly $z_i-z_j$ extra $\bP'$'s between any ${\bf
    S}_i'$ and any ${\bf S}_j'$.  Order the images such that $z_1 \ge
  z_2 \ge \ldots \ge z_m$.
Since $\phi$ is a primitive substitution, $\Omega_\phi$ is
repetitive, so every sufficiently large patch (say, of diameter $D$)
contains at least one copy
of ${\bf S}_1$ and at least one copy of ${\bf S}_m$.

  Now let ${\bf Q}'$ be any patch of size greater than $D$.
  Each patch ${\bf Q}_i$ in $\Omega_\phi$ that
  maps to ${\bf Q}'$ must contain an ${\bf S}_1$ and an ${\bf
    S}_m$.  The images of the ${\bf S}$ patches with maximal and
  minimal values of
  $z_i$ can be identified, in ${\bf Q}'$, by the $z_1-z_m$ extra (or
  missing) ${\bf P}'$'s that occur between them, which can only occur
  between an ${\bf S}_i'$ with $z_i=z_1$ and an ${\bf S}_j'$ with
  $z_j=z_m$.  This means that the locations of patches ${\bf S}_i'$
  with $z_i=z_1$ line up exactly for any two patches ${\bf Q}'$, so
  the number of $\bP'$'s between any two ${\bf Q}'$'s is exactly
  correct.
\end{proof}

We define a \emph{stable pair} as two letters in $\A_\phi$ which are
not eventually coincident and extend this idea to arbitrary
collections of letters: a \emph{stable $m$-tuple} is a collection
$\{a_1, a_2, \dots, a_{m} \}$ of
letters from $\A_\phi$ such that for
all $n$ and $0\le k\le N^n$, $\{ \phi^n(a_1)_k, \phi^n(a_2)_k, \dots,
\phi^n(a_{m})_k \}$ has cardinality $m$. For constant length
substitutions, the coincidence rank, $cr$, is the minimum cardinality
of $\{ \phi^n(a_1)_k, \dots, \phi^n(a_{|\A_\phi|})_k \}$ for $0\le
k\le N^n,n\in \Z$, see \cite{BBK}. Hence $m=cr$ is the maximal
cardinality of a stable $m$-tuple and it follows that the coincidence
ranks of $\phi$ and $\phi'$ are the same. Therefore, to obtain
constraints on $cr$ for homological Pisot substitutions, it suffices
to consider substitutions for which no two letters are strongly
coincident.

\begin{proposition} \label{fraction}
If $\phi$ is a homological Pisot substitution of constant
length $N$ and coincidence rank $cr$, with no strongly coincident pairs of
letters, then there exists a set of letters whose cylinder sets have total
measure $1/cr$.
\end{proposition}

\begin{proof}
  For sufficiently large $n$, there exists a $k$ such that
  $\{\phi^n(a_i)_k\}$, $i \in \{1, \dots, |\A_\phi|\}$ consists of
  exactly $cr$ elements. (Asymptotically, the number of such $k$,
  divided by $N^n$, approaches $1$, \cite{Dek}). Let $n$ and $k$ be as
  described above, and let the set $\{\phi^n(a_i)_{k} \}$ be given by
  $\{b_1, b_2, \dots, b_{cr} \}$.  This set is a stable
  $cr$-tuple. There are then tilings $\bT^i\in
  \Omega_{\phi},i=1,\ldots,cr$, and an $m\in \N$, so that
  $\Phi^m(\bT^i)=\bT^i$ and $0$ is the left endpoint of a tile of type
  $b_i$ in $\bT^i$.  Partition $\A_\phi$ into sets $
  B_j:=\{a:\phi^n(a)_k = b_j\}$, $j \in \{1, \dots, cr \}$. Let $C_j$
  be the union of the cylinder sets corresponding to letters in $B_j$;
  $C_j:= \{\bT: 0\in T_0,T_0\in \bT, T_0$ of type $a$ for some $a\in
  B_j\}$. For fixed $j$ and for each $t\in \R^+\setminus \N$, exactly
  one of the tilings $\bT^i$ is in $C_j$. Thus $1=\lim_{\tau \to
    \infty} \frac{1}{\tau}
  \int_0^{\tau}\sum_{j=1}^{cr}\chi_{C_j}(\bT^j-t)\,dt=\sum_{j=1}^{cr}\lim_{\tau
    \to \infty}  \frac{1}{\tau}
  \int_0^{\tau}\chi_{C_j}(\bT^j-t)\,dt=cr\mu(C_j)$, by unique
  ergodicity of the tiling flow. Thus, for each $j$, the measure of
  the union of the cylinder sets of the elements of $B_j$ is $1/cr$.
\end{proof}

By Theorem~\ref{rationalmeasure}, this implies that $cr$ divides a power of
$N$, thereby proving Theorem~\ref{dim1}, namely that
the Coincidence Rank Conjecture holds for $d=1$.


We finish this section with the proof of Theorem~\ref{not2},
namely
that a homological Pisot
substitution with $d=1$ cannot have $cr=2$.

\begin{proof}If $N$ is odd, this follows from
  Theorem~\ref{dim1}. Suppose, then, that $N$ is even and $cr=2$. We
  will prove that $\dim(H^1)>1$. Every element of $\A_\phi$ must be a
  member of a stable pair by primitivity and the fact that there exist
  $n, k$ such that $\{ \phi^n(a_1)_k, \phi^n(a_2)_k , \dots,
  \phi^n(a_{|\A_\phi|})_k \}$ has cardinality $cr$.  Suppose that $a$
  is in a stable pair with $a'$ and $a''$, $a' \neq a''$.  Then $a'$
  and $a''$ must be strongly coincident, or $cr = 2$ is violated.  By
  identifying strongly coincident letters, and recalling the results
  of Proposition~\ref{factor}, it suffices to consider the case of an
  alphabet of $2m$ letters $\{ a_1, \ldots, a_m, a_1', \ldots,
  a_m'\}$, where $a_i$ and $a_i'$ form a stable pair.

  Ordering the columns and rows of the abelianization by $a_1,
  \ldots, a_m, a_1', \ldots, a_m'$, the matrix takes the form
  ${\mathbf A} =
\begin{pmatrix}A & B \cr B & A
\end{pmatrix}$, where
$A$ and $B$ are $m\times m$ matrices, In particular, the trace of $\mathbf A$
is even.

The space  $G_0^{ER}$ appearing in \eqref{exact sequence}
is the direct limit under
substitution of the graph obtained from all the edges $v_{ij}$ that
describe two-letter words $a_ia_j$ and likewise edges $v_{i'j}$,
$v_{ij'}$ and $v_{i'j'}$.

The symmetry that exchanges each $a_i$ for $a_i'$ maps the components
of $G^{ER}_0$ to each other. If any component is mapped to itself,
then there exists a loop in $G^{ER}_0$, implying that the dimension of
$\check H^1(\Omega_\phi)$ exceeds $d$. To see this, suppose that
$v_{ij}$ and $v_{i'j'}$ are in the same component. Then there exists a
path in the graph connecting these two edges.  Letting $v_{ij} =
v_{i_1 j_1}$, this path may be labeled as
$v_{{i_1}{j_1}}$,$v_{{i_2}{j_1}}$,$v_{{i_2}{j_2}}$, $v_{{i_3}{j_2}}$,
\dots, $v_{i'j'}$.  Since the existence of the word $a_i a_j$ in the
substitution implies the appearance of $a_i' a_j'$ as well, the path
$v_{i'j'} = v_{{i_1'}{j_1'}}$, $v_{{i_2'}{j_1'}}$, $v_{{i_2'}{j_2'}}$,
\dots, $v_{ij}$ exists also, creating a loop in $G^{ER}_0$.

If no component is mapped to itself, then there are an even number of
components in $G_0^{ER}$, which come in symmetric pairs.  Since $\phi$
is a homological Pisot substitution, the dimension of $\check
H^1(\Omega_\phi)$ is one.  This implies that the nonzero eigenvalues
of $\mathbf A^T$ must be $N$ (with multiplicity one) and $1$ (with odd
multiplicity). This is impossible, since $N$ is even and the trace of
$\mathbf A^T$ is even.
\end{proof}

\section{Examples}\label{sec:examples}

In this section we provide examples of homological Pisot substitutions
with $cr>1$. Specifically, all examples have $cr=3$ and all
dilatations have norm divisible by $3$. In each case we first construct
a homological Pisot substitution with $cr=1$ and alphabet
$\A_1=\{A,B,\ldots\}$, and then construct another homological Pisot
substitution with alphabet $\A_2=\{a_1,a_2,a_3,b_1,b_2,b_3,\ldots\}$
that is a triple cover of the first,
and hence has $cr=3$.

\smallskip
\noindent\textbf{Example 1: $d=1$}
\smallskip

We take a $2$-letter alphabet $\A_1 = \{A,B\}$ and a substitution
\begin{equation*}
 \begin{array}{rcl}
\phi_1(A) & = & ABABAAABA \cr
\phi_1(B) & = & BAAABAABA
\end{array}
\quad \text{ with abelianization } \quad \begin{pmatrix} 6 & 3 \cr 6 & 3
\end{pmatrix}
\end{equation*}
Since all substituted letters end in $A$, the first cohomology of
$\Omega_{\phi_1}$ is the direct limit of the abelianization
(see \cite{BD}) and has dimension one.

For $\A_2$, we associate three letters to each letter in $\A_1$, and a
permutation of those three letters to each transition $AA$, $AB$, $BA$
or $BB$.  (Actually, there are no transitions $BB$, but we include a
permutation as a demonstration of the method.) Specifically, to the
transition $AA$ we associate $a_1a_3$, $a_3a_1$ and $a_2a_2$. To $AB$
we associate $a_1b_2$, $a_2b_1$ and $a_3b_3$. To $BA$ we associate $b_1a_3$,
$b_3a_1$ and $b_2a_2$. To $BB$ we associate $b_1b_2$, $b_2b_1$, and $b_3b_3$.
In short, we are allowing transitions
$\{a_1 \hbox{ or } b_1\}\{a_3 \hbox{ or } b_2\}$,
$\{a_2 \hbox{ or } b_2\}\{a_2 \hbox{ or } b_1\}$, and
$\{a_3 \hbox{ or } b_3\}\{a_1 \hbox{ or } b_3\}$.

Note that there are exactly three words in the alphabet $\A_2$
associated to each word in the alphabet $\A_1$. There are three choices
on what the first letter should be, and the rest of the word is
determined by the allowed transitions. These three words disagree at each
point.

The substitution that generates these words is
\begin{equation}\label{example}
\begin{array}{lcr}
\phi_2(a_1) & = & a_1 b_2 a_2 b_1 a_3 a_1 a_3 b_3 a_1 \cr
\phi_2(a_2) & = & a_2 b_1 a_3 b_3 a_1 a_3 a_1 b_2 a_2 \cr
\phi_2(a_3) & = & a_3 b_3 a_1 b_2 a_2 a_2 a_2 b_1 a_3 \cr
\phi_2(b_1) & = & b_1 a_3 a_1 a_3 b_3 a_1 a_3 b_3 a_1 \cr
\phi_2(b_2) & = & b_2 a_2 a_2 a_2 b_1 a_3 a_1 b_2 a_2 \cr
\phi_2(b_3) & = & b_3 a_1 a_3 a_1 b_2 a_2 a_2 b_1 a_3
\end{array}
\quad \begin{array}{l}
\text{with abel-} \\
\text{ ianization }
\end{array} \quad
\begin{pmatrix}
3& 2& 1& 3& 1& 2 \cr
1& 2& 3& 0& 4& 2 \cr
2& 2& 2& 3& 1& 2 \cr
1& 1& 1& 1& 1& 1 \cr
1& 1& 1& 0& 2& 1 \cr
1& 1& 1& 2& 0& 1
\end{pmatrix}
\end{equation}
Note that for each $x_i$ (where $x=a$ or $b$), $\phi(x_i)$ begins with
$x_i$ and ends with $a_i$. This ensures that substituted words will have
the same allowed transitions as the original words.

The abelianization
has eigenvalues $9,1,1,0,0,0$. Since the direct limit
of the transition graph has 3 components, each of which is contractible,
this implies that $\check H^1(\Omega_\phi, \Q)$ is $1$-dimensional.

\smallskip
\noindent\textbf{Example 2: $d=2$}
\smallskip

We take $\A_1$ and $\A_2$ as before, only now $\phi_1$ is a substitution
of degree $2$, with abelianization
$\left ( \begin{smallmatrix} 9 &  6 \cr  6 & 3
\end{smallmatrix} \right )$
 and dilatation $6+3\sqrt{5}$:
\begin{equation*}
\begin{array}{rcl}
\phi_1(A) & = & A B A B A A A B A B A B A B A \cr
\phi_1(B) & = & B A A A B A A B A
\end{array}
\end{equation*}
This is just like the $d=1$ example, only with a $BABABA$ suffix applied
to $\phi_1(A)$. Construct $\phi_2$ similar to \eqref{example}.
This suffix induces a trivial permutation, and so continues to allow us to
have $\phi_2(x_i)$ begin with $x_i$ and end with $a_i$. The abelianization
for $\phi_2$ is now
\begin{equation*}
\begin{pmatrix}
4& 3& 2& 3& 1& 2 \cr
2& 3& 4& 0& 4& 2 \cr
3& 3& 3& 3& 1& 2 \cr
2& 2& 2& 1& 1& 1 \cr
2& 2& 2& 0& 2& 1 \cr
2& 2& 2& 2& 0& 1
\end{pmatrix},
\end{equation*}
and has eigenvalues $6+3\sqrt{5},6-3\sqrt{5},1,1,0,0$.
As before, the direct limit
of the transition graph has $3$ components, each of which is contractible,
so $\check H^1(\Omega_\phi, \Q)$ is $2$-dimensional.

\smallskip
\noindent\textbf{Example 3: $d=3$}
\smallskip

We now take $\A_1 = \{A,B,C\}$, with substitution
\begin{equation*}
\begin{array}{rcl}
\phi_1(A) & = & A B A B A A A B A (C B A)^6 (B A^3)^3 \cr
\phi_1(B) & = & B A A A B A A B A (C B A)^3 (C B A^3)^3 \cr
\phi_1(C) & = & (C B A)^3 (B A^3)^3.
\end{array}
\end{equation*}
with dilatation $\lambda =3 \theta^4$ where $\theta$
is the tribonacci dilatation, \ie the leading root of
$\theta^3 = \theta^2 + \theta + 1$. The
eigenvalues of the abelianization of $\phi_1$ are $\lambda$ and its
algebraic conjugates. It is not hard to see that $\dim(\check H^1(\Omega_\phi, \Q))=3$, and that
the coincidence rank of $\phi_1$ is $1$.

As far as $\phi_2$ and permutations of $\{1,2,3\}$ go,
$C$ is merely a spectator, with transitions
ending in $C$ having a trivial permutation and transitions beginning with
$c_i$ having the same permutations as transitions beginning with $a_i$ or
$b_i$. Specifically, the allowed transitions for $\Omega_{\phi_2}$
are
$\{a_1\hbox{ or } b_1\hbox{ or } c_1 \}\{a_3\hbox{ or } b_2\hbox { or } c_1\}$,
\\
$\{a_2\hbox{ or } b_2\hbox{ or } c_2 \}\{a_2\hbox{ or } b_1\hbox { or } c_2\}$,
and
$\{a_3\hbox{ or } b_3\hbox{ or } c_3 \}\{a_1\hbox{ or } b_3\hbox { or } c_3\}$.

The terms $(C B A)^3$, $(B A^3)^3$ and $(C B A^3)^3$
induce trivial permutations of $\{1,2,3\}$ and give
equal populations of $x_1,x_2,x_3$ to $\phi_2$ of any letter. They are special
cases of the ``padding'' discussed below in Example 4.

The nonzero eigenvalues of the abelianization of $\phi_2$ are $1,
1, \lambda$, and the algebraic conjugates of $\lambda$. Since the eventual
range of the transition graph of $\phi_2$ has $3$ contractible components,
$\dim(\check H^1(\Omega_{\phi_2})) = \dim(\check H^1(\Omega_{\phi_1}))=3$.

\smallskip
\noindent\textbf{Example 4:} Arbitrary $d$
\smallskip

The same tricks can be used to create examples of arbitrary degree
and $cr=3$.

Start with any primitive $d \times d$ Pisot matrix $M_0$ with odd determinant.
Since the determinant is $1 \pmod 2$, some power $M_0^k$ is equal
to the identity ($\bmod\, 2$), and we can assume that $k$ is big enough for
all but one of the eigenvalues of $M_0^k$ to be smaller than $1/3$.
Let $M_1 = 3 M_0^k$ be a new Pisot abelianization on $d$
letters A, B, ..., Z, which is the abelianization of a substitution of the
form
\begin{equation*}
\begin{array}{rcl}
\phi_1(A) & = & A B A B A A A B A \hbox{ plus padding}, \cr
\phi_1(B) & = & B A A A B A A B A \hbox{ plus padding}, \cr
\phi_1(\hbox{any other letter}) & = & \hbox{nothing but padding},
\end{array}
\end{equation*}
where ``padding'' means a product of words, each of the general form
$(V B^{\hbox{\tiny odd }} W A^{\hbox{\tiny odd }} Y)^3$, where $V$, $W$ and $Y$ are
arbitrary words in the letters other than $A$ and $B$. We can always
choose the padding such that $\phi_1$ of any letter begins with that
letter and ends in $A$.

For $\Omega_{\phi_2}$, the allowed transitions are
\{any $x_1$\}\{$a_3$ or $b_2$ or any other $y_1$\},
\{any $x_2$\}\{$a_2$ or $b_1$ or any other $y_2$\}, and
\{any $x_3$\}\{$a_1$ or $b_3$ or any other $y_3$\}.
With these choices, padding yields trivial permutations, and we can
pick $\phi_2(x_i)$ to begin with $x_i$ and end with $a_i$.

The abelianization $M_2$ for $\phi_2$
can be viewed as a collection of $3 \times
3$ blocks, one for each matrix element of the abelianization for
$\phi_1$. All blocks other than the $AA$, $AB$, $BA$ and $BB$ blocks are
multiples of $\left ( \begin{smallmatrix} 1&1&1 \cr 1&1&1 \cr 1&1&1
\end{smallmatrix} \right )$.
This is because ``padding'' yields equal numbers of $x_1$, $x_2$ and $x_3$.
The $AA$, $AB$, $BA$ and $BB$ blocks are the same as in our $d=1$ example,
plus multiples of $\left ( \begin{smallmatrix} 1&1&1 \cr 1&1&1 \cr 1&1&1
\end{smallmatrix} \right )$.

The populations of $\phi_2(a_1)$ and $\phi_2(a_3)$ add up to twice the
population of $\phi_2(a_2)$, and the populations of $\phi_2(b_1)$ and
$\phi_2(b_2)$ add up to twice the population of $\phi_2(b_3)$.  For $X
\ne A$ or $B$, the populations of $\phi_2(x_1)$, $\phi_2(x_2)$ and
$\phi_2(x_3)$ are all the same. This yields $2d-2$ vectors
in the kernel of
$M_2$, so the rank of $M_2$ is at most $d+2$. However, every
eigenvalue of $M_1$ is also an eigenvalue of $M_2$. Since the eventual
range of the transition graph of $\phi_2$ has three components, $1$ must
be an eigenvalue of $M_2$ with multiplicity 2. Together, these imply
that $\dim(\check H^1(\Omega_{\phi_2}))=d$.

\bigskip\noindent{\bf Acknowledgments.}
The work of L.S.\ is partially supported by
the National Science Foundation.
M.B.\ and H.B.\ thank the Mathematisches Forschungsinstitut Oberwolfach
for its hospitality during
the Research in Pairs Programme January 11-24, 2009, and in this context
we are also grateful to Sonja \v{S}timac.
H.B.\ thanks Delft University of Technology as well for its support
during the summer of 2009.


\bigskip

{\Small {\parindent=0pt Department of Mathematics, Montana State University,
Bozeman, MT 59717, USA \\
barge@math.montana.edu \\
\phantom{asdf} \\
Department of Mathematics,
University of Surrey, Guildford, Surrey GU2 7XH, UK \\
h.bruin@surrey.ac.uk\\
\phantom{asdf} \\
Department of Mathematics, University of Arizona, Tucson, AZ  85724 USA
\\ljones@math.arizona.edu\\ \phantom{asdf} \\
Department of Mathematics, University of Texas, Austin, TX 78712, USA \\
sadun@math.utexas.edu
}}


\begin{thebibliography}{99}
\bibitem[A]{Aus} J.\ Auslander, Minimal flows and their extensions,
  {\em North-Holland Mathematical Studies}, vol. 153, North-Holland,
  Amsterdam, New York, Oxford, and Tokyo, (1988).

\bibitem[AP]{AP} J.\ E.\ Anderson and I.\ F.\ Putnam, Topological invariants
  for substitution tilings and their associated $C^*$-algebras, {\em
    Ergodic Theory \& Dynamical Systems} \textbf{18} (1998), 509--537.

\bibitem[BBK]{BBK} V.\ Baker, M.\ Barge and J.\ Kwapisz, Geometric
  realization and coincidence for reducible non-unimodular Pisot
  tiling spaces with an application to $\beta$-shifts,
  \textit{J. Instit. Fourier.}  {\bf56} (7) (2006), 2213-2248.

\bibitem[BD1]{BDproper} M.\ Barge and  B.\ Diamond, A complete invariant for the topology of one-dimensional substitution tiling spaces, {\em Ergodic Theory \& Dynamical Systems} {\bf 21} (2001), 1333-1358.

\bibitem[BD2]{BD} M.\ Barge and  B.\ Diamond, Coincidence for substitutions of
Pisot type, {\em Bulletin de la Soci\'{e}t\'{e} Math\'{e}matique de France} {\bf
130} (2002), 619-626.

\bibitem[BD3]{BDcohomology} M.\ Barge and B.\ Diamond, Cohomology in
  one-dimensional substitution tiling spaces, {\em
    Proc. Amer. Math. Soc.} {\bf 136} (6) (2008), 2183-2191.

\bibitem[BK]{BK} M.\ Barge and J.\ Kwapisz, Geometric theory of
  unimodular Pisot substitutions, \textit{Amer J. Math.} {\bf128}
  (2006), 1219-1282.

\bibitem[BS]{BS} V.\ Berth\'{e} and A.\ Siegel, Tilings associated with
  beta-numeration and substitutions, {\em Integers: electronic journal
    of combinatorial number theory} {\bf 5} (2005), A02.

\bibitem[BSw]{BSw} M.\ Barge and R.\ Swanson, Rigidity in
  one-dimensional tiling spaces, \textit{Top.\ and its Appl.}
  {\bf157} (17) (2007), 3095-3099.

\bibitem[BT]{BT} E.\ Bombieri and J.\ E.\ Taylor, Which distributions of
  matter diffract? An initial investigation, {\em J.\ Physique} {\bf
    47} (7, suppl. Colloq. C3) C3-19-C3-28, International workshop on
  aperiodic crystals (les Houches) (1986).

\bibitem[CS]{CS} V.\ Canterini and A.\ Siegel, Geometric representation
  of substitutions of Pisot type, {\em Trans.\ Amer.\
    Math.\ Soc.} {\bf 353} (2001), 5121-5144.

\bibitem[De]{Dek} F.\ M.\ Dekking, The spectrum of dynamical systems
  arising from substitutions of constant length, {\em
    Z. Wahrscheinlichkeitstheorie verw. Gebiete} {\bf 41} (1978),
  221-239.

\bibitem[Dw]{Dworkin} S.\ Dworkin, Spectral theory and x-ray diffraction,
{\em J.\ Math.\ Phys.} {\bf 34} (1993), 2964--2967.

\bibitem[HS]{HS} M.\ Hollander and B.\ Solomyak, Two-symbol Pisot
  substitutions have pure discrete spectrum, {\em Ergodic Theory \&
    Dynamical Systems} {\bf 23} (2003), 533-540.

\bibitem[LMS]{LMS} J.-Y.\ Lee, R.\ V.\ Moody and B.\ Solomyak, Pure point
  dynamical and diffraction spectra, {\em Annales Henri Poincar\'{e}}
  {\bf3} (2002), 1003-1018.

\bibitem[M]{recog} B.\  Moss\'{e}, Puissances de mots et
  reconnaissabilit\'{e} des points fixes d'une substitution, {\em
    Theoretical Computer Science} {\bf 99} (1992), 327--334.

\bibitem[IR1]{ItoRao} S.\ Ito and H.\ Rao, Atomic surfaces, tilings and
  coincidences I. Irreducible case, {\em Israel J. Math.} {\bf 153}
  (2006), 129-156.

\bibitem[IR2]{ItoRao2} S.\ Ito and H.\ Rao, Atomic surfaces, tilings and
  coincidences II. Reducible case, {\em Ann. Inst. Fourier, Grenoble}
  {\bf56} (6) (2006), 2285-2313.

\bibitem[Kel]{Kellendonk} J.\ Kellendonk,
Pattern-equivariant functions and cohomology,
{\em J. Phys. A.} {\bf 36} (2003), 1--8.

\bibitem[KP]{KP} J.\ Kellendonk and I.\ Putnam, The Ruelle-Sullivan map for
$\R^n$-actions, {\em Math. Ann.} {\bf 334} (2006), 693--711.

\bibitem[Sa1]{Sadun} L.\ Sadun, Pattern-equivariant cohomology with
  integer coefficients. {\em Ergodic Theory \& Dynamical Systems} {\bf
      27} (2007), 1991--1998.
      
\bibitem[Sa2]{Sadun2} L.\ Sadun, Exact regularity and the cohomology of tiling spaces, arXiv:math/1004.2281 $ \langle$http://arxiv.org/abs/1004.2281v1$ \rangle$, 1--15.


\bibitem[So1]{solomyak1} B.\ Solomyak, Eigenfunctions for substitution tiling systems,\emph{ Advanced Studies in Pure Mathematics}, \textbf{49} (2007), 433--454.

\bibitem[So2]{solomyak} B.\ Solomyak, Nonperiodicity implies unique composition for self-similar translationally finite tilings, {\em Discrete and
Computational Geometry} {\bf 20} (2) (1998), 265--279.

\end{thebibliography}
\end{document}